\documentclass[fontsize=10pt,fleqn,a4paper]{scrartcl}

\usepackage[english]{babel}
\usepackage[T1]{fontenc}
\usepackage[utf8]{inputenc}

\usepackage{amsmath,  
            amssymb,  
            amsthm,   
            amsfonts  
            }

\usepackage{mathtools}

\usepackage{braket}   

\usepackage{extpfeil} 

\newcommand{\IN}{\mathbb{N}}
\newcommand{\IZ}{\mathbb{Z}}

\newcommand{\isomorphic}{\cong}

\DeclareMathOperator{\Hom}{Hom}

\DeclareMathOperator{\Irr}{Irr}

\DeclareMathOperator{\ord}{ord}

\newcommand{\captionstringtheorem}{Theorem}
\newcommand{\captionstringlemma}{Lemma}
\newcommand{\captionstringcorollary}{Corollary}
\newcommand{\captionstringlemmadef}{Lemma and definition}
\newcommand{\captionstringtheoremdef}{Theorem and definition}
\newcommand{\captionstringdefinition}{Definition}
\newcommand{\captionstringproposition}{Proposition}
\newcommand{\captionstringexample}{Example}
\newcommand{\captionstringconjecture}{Conjecture}
\newcommand{\captionstringconvention}{Convention}

\newtheoremstyle{dotless} 
			{\bigskipamount}    
			{\bigskipamount}    
			{\nopagebreak}      
			{}                  
			{\bfseries}         
			{:}                 
			{\newline}          
			{}                  
\newtheoremstyle{dotless2} 
			{\bigskipamount}    
			{0.0em}             
			{}                  
			{}                  
			{\bfseries}         
			{:}                 
			{0.5em}             
			{}                  
\swapnumbers

\newcounter{theoremnumber}
\numberwithin{theoremnumber}{section}

\theoremstyle{dotless}
\newtheorem{theorem}[theoremnumber]{\captionstringtheorem}

\newtheorem{lemma}[theoremnumber]{\captionstringlemma}

\newtheorem{definition}[theoremnumber]{\captionstringdefinition}
\newtheorem{example}[theoremnumber]{\captionstringexample}
\newtheorem{conjecture}[theoremnumber]{\captionstringconjecture}

\newtheorem*{convention}{\captionstringconvention}

\theoremstyle{dotless2}
\newtheorem{remark}[theoremnumber]{}

\allowdisplaybreaks[1]

\usepackage{textcomp}

\usepackage[normalem]{ulem} 

\usepackage{enumitem}

\usepackage{csquotes}

\usepackage[numbers]{natbib}
\usepackage[
	pdfpagelabels=true,
	plainpages=false
	]{hyperref}
	
\hypersetup{
colorlinks=true,
linkcolor=red,
urlcolor=blue,
citecolor=blue,
linktocpage=true, 
pdfpagelayout={OneColumn},
pdfstartview= 
}

\newcommand{\titleOfThisDocument}{A Reduction theorem for the $W$-graph decomposition conjecture}

\author{Johannes Hahn}
\title{\titleOfThisDocument}

\hypersetup{
pdfinfo=
	{  
		Title={\titleOfThisDocument},
		Author={Johannes Hahn},
		Keywords={Coxeter group, Hecke algebra, W-graph, W-graph algebra, decomposition conjecture},
		Subject={Representation theory}
	}
}

\begin{document}
\maketitle

\begin{abstract}
Let $W$ be a finite Coxeter group and $\Omega$ be its $W$-graph algebra as defined by Gyoja. The author's previous paper \cite{hahn2016wgraphs} considered this algebra in some detail, proposed, and proved in some small cases the $W$-graph decomposition conjecture.
The purpose of the current paper is to prove a reduction theorem for (a slightly stronger version of) that conjecture to indecomposable Coxeter groups in the sense that the conjecture is true for $W=W_1\times W_2$ if it holds for $W_1$ and $W_2$.
\end{abstract}

\section{Introduction}

Let $(W,S)$ be a finite Coxeter group and $\Omega=\Omega(W,S)$ its $W$-graph algebra as first introduced by Gyoja in \cite{gyoja}. The author's first paper \cite{hahn2016wgraphs} contained a more detailed investigation this algebra and its properties. Its distinguishing feature is that its modules correspond (up to some minor technical details) to $W$-graphs if $W$ is finite: Every $W$-graphs induces an $\Omega$-module and after choosing an appropriate basis every $\Omega$-module is a $W$-graph.

One of the main results of \cite{hahn2016wgraphs} is that $\Omega$ is actually a quotient of the path algebra over the so called compatibility graph $\mathcal{Q}_W$ (as defined in \cite{stembridge2008admissble}) with some very explicit relations. In particular there is a natural decomposition of the unit of $\Omega$ into a sum of orthogonal idempotents corresponding to the vertices in $\mathcal{Q}_W$.

The $W$-graph decomposition conjecture proposes that there should be another, but compatible decomposition into idempotents which induces a natural filtration on every $\Omega$-module such that the quotients of this filtration are semisimple and contain irreducibles of only one isomorphism class. In this sense this filtration realizes the decomposition of the module into its irreducible constituents in a natural way, hence the name of the conjecture.

Although the conjecture seems to hold for many types of Coxeter group (including all of type $A$) so far it has only been proven by somewhat pedestrian means for the types $I_2(m)$, $A_1$--$A_4$ and $B_3$ in \cite{hahn2016wgraphs}. The purpose of the current paper is to prove a general result saying that (a variation of) the decomposition conjecture holds for a decomposable Coxeter group $W=W_1\times W_2$ if it holds for $W_1$ and $W_2$. For technical reasons a strengthening of the conjecture is necessary for the proof to work. It turns out that the explicit constructions in the previous proofs all prove this stronger form of the conjecture as well so that the results from the previous paper still apply.

This paper is organised as follows: Section \ref{section:preparations} recalls the definitions of Iwahori-Hecke algebras, $W$-graph algebras and certain constructions from \cite{hahn2016wgraphs}.

Section \ref{section:decomposition_conjecture} recalls the decomposition conjecture and states the stronger form needed for this paper. It also introduces the subalgebra $\Psi\subseteq\Omega$ which will be needed for the proof.

Section \ref{section:maintheorem} then states and proves the main theorem of this paper that the strong form of the decomposition conjecture holds for $W=W_1\times W_2$ if it holds for $W_1$ and $W_2$.

\section{Preparations}\label{section:preparations}

\begin{convention}
For the rest of the paper fix a (not necessarily finite) Coxeter group $(W,S)$ and a \enquote{good} ring in the sense of \citep[definition 1.5.9]{geckjacon}, i.e. $2\cos(\frac{2\pi}{\ord(st)})\in k$ for all $s,t\in S$.
\end{convention}

\begin{definition}[c.f. \cite{geckjacon}]
The \emph{Iwahori-Hecke algebra} $H=H(W,S)$ is the $k[v^{\pm 1}]$-algebra  which is freely generated by $(T_s)_{s\in S}$ subject only to the relations
\[T_s^2 = 1 + (v-v^{-1}) T_s\quad\textrm{and}\]
\[\underbrace{T_s T_t T_s \ldots}_{m_{st}\,\text{factors}} = \underbrace{T_t T_s T_t \ldots}_{m_{st}\,\text{factors}}\]
where $m_{st}$ denotes the order of $st\in W$.

For each parabolic subgroup $W_J\leq W$ the Hecke algebra $H(W_J,J)$ will be identified with the \emph{parabolic subalgebra} $H_J := \operatorname{span}_{k[v^{\pm1}]}\Set{T_w | w\in W_J}\subseteq H$.
\end{definition}

\begin{definition}[c.f. \cite{KL} and \cite{geckjacon}]
A \emph{$W$-graph with edge weights in $k$} is a triple $(\mathfrak{C},I,m)$ consisting of a finite set $\mathfrak{C}$ of \emph{vertices}, a \emph{vertex labelling} map $I: \mathfrak{C}\to \{J \mid J\subseteq S\} $ and a family of \emph{edge weight} matrices $m^s\in k^{\mathfrak{C}\times\mathfrak{C}}$ for $s\in S$ such that the following conditions hold:
\begin{enumerate}
	\item $\forall x,y\in\mathfrak{C}: m_{xy}^s \neq 0 \implies s\in I(x)\setminus I(y)$.
	\item The matrices
	\[\omega(T_s)_{xy} := \begin{cases} -v^{-1}\cdot 1_k & \textrm{if}\;x=y, s\in I(x) \\ v\cdot 1_k & \textrm{if}\;x=y, s\notin I(x) \\ m_{xy}^s & \textrm{otherwise}\end{cases}\]
	induce a matrix representation $\omega: H\to k[v^{\pm1}]^{\mathfrak{C}\times\mathfrak{C}}$.
\end{enumerate}
The associated directed graph is defined as follows: The vertex set is $\mathfrak{C}$ and there is a directed edge $x\leftarrow y$ if and only if $m_{xy}^s\neq 0$ for some $s\in S$. If this is the case, then the value $m_{xy}^s$ is called the \emph{weight} of the edge (and in fact does not depend on $s$). The set $I(x)$ is called the \emph{vertex label} of $x$.
\end{definition}

\begin{remark}
Note that condition a. and the definition of $\omega(T_s)$ already guarantees $\omega(T_s)^2=1+(v-v^{-1})\omega(T_s)$ holds so that the only non-trivial requirement in condition b. is the braid relation $\omega(T_s)\omega(T_t)\omega(T_s)\ldots = \omega(T_t)\omega(T_s)\omega(T_t)\ldots$.
\end{remark}

Given a $W$-graph as above the matrix representation $\omega$ turns $k[v^{\pm1}]^{\mathfrak{C}}$ into a $H$-module. It is natural to ask whether a converse is true. In many situations the answer is yes as shown by Gyoja.
\begin{theorem}[c.f. \cite{gyoja}]
Let $W$ be finite, $K\subseteq\mathbb{C}$ be a splitting field for $W$. Then every irreducible representation of $K(v)H$ can be realized as a $W$-graph module for some $W$-graph with edge weights in $K$.
\end{theorem}

\begin{remark}
Gyoja also provides an example of a finite-dimensional representation of the affine Weyl group of type $\tilde{A}_n$ that is not induced by a $W$-graph.
\end{remark}

\begin{remark}
An auxiliary object in Gyoja's proof and the object of interest in this and the author's previous papers is the so called $W$-graph algebra.
\end{remark}

\begin{definition}[{Gyoja's $W$-graph algebra (c.f. \citep[definition 2.4]{gyoja})}]
Consider the free algebra $\IZ\langle e_s, x_s | s\in S\rangle$. Define
\[\iota(T_s) := -v^{-1} e_s + v (1-e_s) + x_s \in \IZ\langle e_s, x_s \rangle[v^{\pm1}]\]
for all $s\in S$ and write their \emph{braid-commutator}
\[\Delta_m(\iota(T_s),\iota(T_t)) :=  \underbrace{\iota(T_s)\iota(T_t)\iota(T_s)\ldots}_{m_{st}\,\text{factors}} - \underbrace{\iota(T_t)\iota(T_s)\iota(T_t)\ldots}_{m_{st}\,\text{factors}}\]
as
\[\sum_{\gamma\in\IZ} y^\gamma(s,t) v^\gamma\]
for some (uniquely determined) $y^\gamma(s,t)\in\IZ\langle e_s, x_s \rangle$.

\medbreak
Define $\Omega$ to be the quotient of $\IZ\langle e_s, x_s \rangle$ modulo the relations
\begin{enumerate}
	\item $e_s^2=e_s$, $e_s e_t = e_t e_s$,
	\item $e_s x_s=x_s$, $x_s e_s = 0$,
	\item $y^\gamma(s,t) = 0$
\end{enumerate}
for all $s,t\in S$ and $\gamma\in\IZ$.

$\Omega=\Omega(W,S)$ is called the \emph{$W$-graph algebra}.
\end{definition}

\begin{remark}
This is a definition in the equal parameter case. There are more general definitions of $W$-graph algebras that cover the case of non-equal parameters. The author has previously given one such definition in \cite{hahn2017canon}.
\end{remark}

\begin{remark}
The definition immediately implies that $T_s\mapsto\iota(T_s)$ defines a homomorphism of $k[v^{\pm1}]$-algebras $\iota: H\to k[v^{\pm1}]\Omega$. In fact this is an embedding as shown in \citep[corollary 10]{hahn2016wgraphs}.

$W$-graph algebras have the distinguishing feature that each $W$-graph $(\mathfrak{C},I,m)$ with edge weights in $k$ not only defines the structure of a $H$-module on $k[v^{\pm1}]^{\mathfrak{C}}$ but that it induces a canonical $k\Omega$-module structure on $k^\mathfrak{C}$ via
\[e_s \cdot \mathfrak{z} := \begin{cases} \mathfrak{z} & s\in I(\mathfrak{x}) \\ 0 & \text{otherwise}\end{cases}\]
\[x_s \cdot \mathfrak{z} := \sum_{\mathfrak{x}\in\mathfrak{C}} m_{\mathfrak{xz}}^s \mathfrak{x}\]
for all $\mathfrak{z}\in\mathfrak{C}$. Then $k[v^{\pm1}]^{\mathfrak{C}\times\mathfrak{C}}$ is a $k[v^{\pm1}]\Omega$-module and restriction to a $H$-module gives back the $H$-module in the definition.

Conversely if $V$ is a $k\Omega$-module that has a $k$-basis $\mathfrak{C}$ w.r.t. which all $e_s$ act as diagonal matrices, then $V$ is obtained from a $W$-graph $(\mathfrak{C},I,m)$ in this way. In this way one can interpret $\Omega$-modules as $W$-graphs up to choice of a basis.
\end{remark}

\begin{example}
The trivial group is a Coxeter group $(1,\emptyset)$ and its associated $W$-graph algebra is just $\IZ$.

A cyclic group of order 2 is a Coxeter group $(\set{1,s},\set{s})$ of rank 1 and its associated $W$-graph algebra is free as a $\IZ$-module with basis $\Set{e_s, 1-e_s,x_s}$. The multiplication of the basis elements is completely determined by the relations because $x_s x_s = x_s (e_s x_s) = (x_s e_s) x_s = 0$.
\end{example}

\begin{remark}
Being defined by generators and relations, $\Omega$ satisfies a universal mapping property which we will now state in a slightly different, more convenient form.
\end{remark}

\begin{lemma}\label{omega:universal_mapping_property}
Consider the category of rings. Then pre-composing with the quotient map $\IZ\langle e_s, x_s \mid s\in S\rangle\twoheadrightarrow\Omega$ is a natural isomorphism
\[\Hom(\Omega, A) \isomorphic \Big\lbrace (\tilde{e}_s, \tilde{x}_s)_{s\in S} \in \prod_{s\in S} A^2 \Big| \begin{array}{l}
\tilde{e}_s,\tilde{x}_s\text{ statisfy relations a. and b. and} \\
\text{the induced map }\IZ\langle e_s, x_s\rangle[v^{\pm1}] \to A[v^{\pm1}] \\
\text{annihilates } \Delta_{m_{st}}(\iota(T_s),\iota(T_t)) \text{ for all }s,t\in S\end{array} \Big\rbrace.\]
\end{lemma}
\begin{proof}
See \citep[Lemma 2.4]{hahn2016wgraphs}.
\end{proof}

\begin{example}[Parabolic morphisms]
A trivial consequence of this lemma is that for every $J\subseteq S$ the \emph{parabolic morphism} $\Omega(W_J,J) \to \Omega(W,S), e_s\mapsto e_s, x_s\mapsto x_s$ is well-defined. It was proven in \cite{hahn2017canon} that this is always an embedding.
\end{example}

\subsection{$\Omega$ as a quotient of a path algebra}

The following definition already appears in Gyoja's paper (c.f. \citep[definition 2.5]{gyoja}) although with a different notation:
\begin{definition}
In $\Omega$ define the following elements for all $I,J\subseteq S, s\in S$:
\[E_I:=\Big(\prod_{t\in I}e_t\Big)\Big(\prod_{t\in S\setminus I}(1-e_t)\Big)\]
\[X_{IJ}^s:=E_I x_s E_J\]
\end{definition}

\begin{definition}[{c.f. \citep[section 4]{stembridge2008admissble}}]
The \emph{compatibility graph} of $(W,S)$ is the directed graph $\mathcal{Q}_W$ with vertex set $\left\lbrace I \mid I\subseteq S\right\rbrace$ and a single edge $I\leftarrow J$ if and only if $I\setminus J\neq\emptyset$ and no element of $I\setminus J$ commutes with an element of $J\setminus I$.

An edge $I\leftarrow J$ with $I\supseteq J$ is called an \emph{inclusion edge}, all other edges (i.e. those with $J\setminus I\neq\emptyset$) are called \emph{transversal edges}.

Note that transversal edges always come in pairs: If $I\leftarrow J$ is transversal, then there is a transversal edge $I\rightarrow J$ as well. The statement that $I,J\subseteq S$ are connected by a pair of transversal edges will be denoted by $I \leftrightarrows J$.
\end{definition}

\begin{remark}
One of the main goals of the author's first paper was to establish that $\Omega$ is in fact a quotient of the path algebra $\IZ\mathcal{Q}_W$ such that $E_I$ is the (image of the) element corresponding to the vertex $I$ and $X_{IJ}^s$ is the (image of the) element corresponding to the edge $I\leftarrow J$ in $\mathcal{Q}_W$. In particular $X_{IJ}^s$ doesn't depend on $s$, only on $I$ and $J$.
\end{remark}

\begin{theorem}[{c.f. \citep[theorem 13]{hahn2016wgraphs}}]
Define polynomials $\tau_r\in\IZ[T]$ by the following recursion:
\[\tau_{-1} :=0,\quad \tau_0 := 1, \quad \tau_r := T\tau_{r-1}-\tau_{r-2}\]
and let $a_{r,i}$ denote its coefficients, i.e. $\tau_r = T^r + a_{r,m-1} T^{m-1} + \ldots + a_{r,1} T + a_{r,0}$.

\medskip
For all $I,J\subseteq S$, $s,t\in S$ and $r\in\IN$ define
\[P_{IJ}^r(s,t) :=  E_I \underbrace{x_s x_t x_s\ldots}_{r\,\textrm{factors}} E_J = \begin{cases} E_I E_J & r=0 \\ \sum\limits_{I_1,\ldots,I_{r-1}\subseteq S} X_{I I_1}^s X_{I_1 I_2}^t X_{I_2 I_3}^s \ldots X_{I_{r-1} J}^s & r>0, 2\nmid r \\ \sum\limits_{I_1,\ldots,I_{r-1}\subseteq S} X_{I I_1}^s X_{I_1 I_2}^t X_{I_2 I_3}^s \ldots X_{I_{r-1} J}^t & r>0, 2\mid r \end{cases}.\]

With this notation $\Omega$ is freely generated by $\Set{E_I, X_{IJ}^s | I,J\subseteq S, s\in S}$ modulo the following relations:
\begin{enumerate}
	\item Quiver-relations:
	\begin{enumerate}
		\item $E_I E_J = \delta_{IJ} E_I$ and $\sum_{I\subseteq S} E_I = 1$,
		\item $E_I X_{IJ}^s = X_{IJ}^s = X_{IJ}^s E_J$,
		\item $X_{IJ}^s = 0$ if $s\notin I\setminus J$,
	\end{enumerate}
	
	\item Additional relations for all $s,t\in S$ with $m=\ord(st)<\infty$:
	\begin{enumerate}
	\item[$(\alpha^{st})$] The relations
	\[0 = \sum_{k=0}^{m-1} a_{m-1,k} P_{IJ}^k(s,t)\]
	for all $I,J\subseteq S$ if either
	\begin{itemize}
		\item $s\in I$, $t\notin I$, $s\in J$, $t\notin J$ and $2\mid m_{st}$ or
		\item $s\in I$, $t\notin I$, $s\notin J$, $t\in J$ and $2\nmid m_{st}$
	\end{itemize}
	holds.
	\item[$(\beta^{st})$] For all $I,J\subseteq S$ with $s,t\in I\setminus J$ the relations
	\[0=P_{IJ}^1(s,t) - P_{IJ}^1(t,s)=P_{IJ}^2(s,t) - P_{IJ}^2(t,s)=\ldots= P_{IJ}^m(s,t) - P_{IJ}^m(t,s).\]
	\end{enumerate}
\end{enumerate}
\end{theorem}

\begin{remark}
Note that $(\beta)$ includes the relation $X_{IJ}^s = X_{IJ}^t$ for all $s,t\in I\setminus J$. This justifies to write $X_{IJ}$ in situation where we don't care which $s\in I\setminus J$ we choose.

Furthermore $(\alpha)$ includes the relation $X_{IJ}^s=0$ if there exists a $t\in J\setminus I$ such that $\ord(st)=2$. Therefore $\Omega$ really is a quotient of the path algebra $\IZ\mathcal{Q}_W$.
\end{remark}


\begin{remark}
In terms of this second set of generators, the parabolic morphisms $j: \Omega(W_J,J)\to\Omega(W,S)$ do the following:
\[j(E_A) = \sum_{\substack{A'\subseteq S \\ A' \cap J = A}} E_{A'}\]
\[j(X_{AB}^s) = \sum_{\substack{A',B'\subseteq S \\ A' \cap J = A, B'\cap J=B}} X_{A'B'}^s\]
\end{remark}

\section{The decomposition conjecture}\label{section:decomposition_conjecture}

\begin{conjecture}[{$W$-graph decomposition conjecture (c.f. \cite{hahn2016wgraphs})}]
Let $k$ be a good ring for $W$. There exists a family $(F^\lambda)_{\lambda\in\Irr(W)}$ of elements of $k\Omega$ such that:
\begin{enumerate}[label=(Z{\arabic*})]
	\item The $F^\lambda$ constitute are a decomposition of the identity into orthogonal idempotents:
	\[\forall\lambda,\mu\in\Irr(W): F^\lambda F^\mu = \delta_{\lambda\mu} F^\lambda, \quad 1=\smash{\sum_{\lambda\in\Irr(W)}} F^\lambda\]
	\item This decomposition is compatible with the decomposition induced by the path-algebra structure:
	\[\forall\lambda\in\Irr(W) \forall I\subseteq S: E_I F^\lambda = F^\lambda E_I\]
	\item There is a partial order $\preceq$ on $\Irr(W)$ such that only \enquote{downward edges} exist: If $F^\lambda \Omega F^\mu \neq 0$, then $\lambda \preceq \mu$.
	\item There are surjective $k$-algebra morphisms $\psi_\lambda: k^{d_\lambda\times d_\lambda} \twoheadrightarrow F^\lambda kW F^\lambda$ for all $\lambda\in\Irr(W)$ where $d_\lambda$ denotes the degree of the character $\lambda$.
\end{enumerate}
\end{conjecture}

\begin{remark}
In \cite{hahn2016wgraphs} it was already proven that the Coxeter groups of types $A_1 - A_4$, $I_2(m)$ and $B_3$ all satisfy this decomposition conjecture. For the purpose of the reduction theorem for products, a sharper version of the conjecture will be needed:
\end{remark}

\begin{definition}
$\Psi=\Psi(W,S)$ is the subalgebra of $\Omega(W,S)$ generated by all $E_I$ and all \emph{transversal} edges $X_{IJ}$.
\end{definition}

\begin{conjecture}[Strong $W$-graph decomposition conjecture]
Using the notations of the decomposition conjecture, the strong decomposition conjecture for $(W,S)$ claims that the following additional statements are true
\begin{enumerate}[label=(Z{\arabic*}), start=5]
	\item $F^\lambda X_{IJ}F^\lambda\in\Psi(W,S)$ for \emph{all} edges $I\leftarrow J$.
	\item $F^\lambda\in\Psi(W,S)$.
\end{enumerate}
\end{conjecture}

\begin{remark}
Inspection of the proofs of the decomposition conjecture for $A_1-A_4$, $B_3$ and $I_2(m)$ from \cite{hahn2016wgraphs} shows that the strong decomposition conjecture is also true in these cases.
\end{remark}

\begin{remark}
Note that parabolic morphisms map $\Psi(W_J,J)$ into $\Psi(W,S)$.
\end{remark}
\section{A reduction theorem for products}\label{section:maintheorem}

\begin{convention}
In this section fix a reducible Coxeter group $W=W_1\times W_2$. To indicate whether something refers to $W_1$, $W_2$ or $W$ indices will be used: $1$, $2$ and no index respectively. Therefore $\Omega$ is short for $\Omega(W,S)$, $H_2$ is short for $H(W_2,S_2)$, the set of simple reflections $S$ decomposes as $S=S_1\sqcup S_2$ etc.

\medbreak
The parabolic inclusions $W_1\to W$, $W_2\to W$ and all induced parabolic inclusions will be denoted by $\iota_1$ and $\iota_2$ respectively.

\medbreak
For subsets $I\subseteq S$ the notation $I=X\sqcup Y$ will be used the refer to the canonical partition with $X\subseteq W_1$, $Y\subseteq W_2$. Often we will also write $I=I_1\sqcup I_2$ in accordance with the index-convention.
\end{convention}

\begin{remark}
The first observation we make and will use through-out the following proofs without further mention: All non-zero $X_{IJ}\in\Omega$ with $I=I_1\sqcup I_2$, $J=J_1\sqcup J_2$ have one of the following forms
\begin{enumerate}
	\item Inclusion edge: $I_1\supseteq J_1$, $I_2\supseteq J_2$ at least one of which is a proper inclusion.
	\item Transversal edges: One of these two cases
	\begin{enumerate}
	\item $I_1 \leftrightarrows J_1$ in $\mathcal{Q}_{W_1}$ and $I_2=J_2$.
	\item $I_1=J_1$ and $I_2 \leftrightarrows J_2$ in $\mathcal{Q}_{W_2}$.
	\end{enumerate}
\end{enumerate}
This follows immediately from the condition of compatibility: If $I\leftrightarrows J$ in $\mathcal{Q}_W$, then all $s\in I\setminus J$ are connected to all $t\in J\setminus I$ in the Dynkin-diagram of $(W,S)$ so that the symmetric set difference $I\Delta J$ is completely contained either in $S_1$ or in $S_2$.
\end{remark}

\subsection{Tensor products}

\begin{remark}
For Hecke algebras the reduction to the irreducible case is easy because $H=H_1\otimes_{k[v^{\pm1}]} H_2$ canonically. This, as it turns out, is not the case for $W$-graph algebras.
\end{remark}

\begin{lemma}\label{tensor_morphism}
Let $(W,S)$ be a reducible Coxeter group as above.
	
There is a unique morphism $\tau: \Omega\to \Omega_1\otimes_\IZ \Omega_2$ with 
\[e_s \mapsto \begin{cases} e_s \otimes 1 & \text{if }s\in S_1 \\ 1\otimes e_s & \text{if }s\in S_2\end{cases} \quad\text{and}\quad
x_s \mapsto \begin{cases} x_s \otimes 1 & \text{if }s\in S_1 \\ 1\otimes x_s & \text{if }s\in S_2\end{cases}\]
This morphism satisfies
\begin{enumerate}
	\item $\tau(E_I) = E_{I_1} \otimes E_{I_2}$.
	\item $\displaystyle\tau(X_{IJ}^{s}) = \begin{cases} X_{I_1 J_1}^{s} \otimes E_{I_2}E_{J_2} & \text{if }s\in S_1 \\ E_{I_1}E_{J_1} \otimes X_{I_2 J_2}^{s} & \text{if }s\in S_2\end{cases}$.
	\item $\displaystyle\tau(\iota(T_s)) = \begin{cases} \iota_1(T_s)\otimes 1 & \text{if }s\in S_1 \\ 1\otimes\iota_2(T_s) & \text{if }s\in S_2\end{cases}$,
	
	i.e. $\tau$ restricts to the canonical isomorphism $H \to H_1 \otimes H_2$.
	\item $\tau$ maps $\Psi$ onto $\Psi_1\otimes\Psi_2$.
\end{enumerate}
\end{lemma}
\begin{proof}
The well-definedness of $\tau$ is an easy consequence of its definition and lemma \ref{omega:universal_mapping_property}.
\end{proof}

\begin{remark}
Contrary to the Hecke algebra case this morphism need not be injective as the following theorem shows.
\end{remark}

\begin{theorem}\label{tensor_morphism:kernel}
$\ker(\tau)$ is generated as a two-sided ideal by $\Set{X_{IJ}\in\Omega \mid I_1\supsetneq J_1 \wedge I_2\supsetneq J_2}$.
\end{theorem}
\begin{proof}
The tensor product $\Omega_1\otimes\Omega_2$ is obtained from $\Omega$ by forcing the generators $e_s, x_s$ with $s\in S_1$ to commute with the generators $e_t, x_t$ with $t\in S_2$ and $\tau$ is the quotient map $\Omega\to\Omega_1\otimes\Omega_2$ obtained by adding these relations.

Since $[e_s,e_t]=0$ already holds in $\Omega$, this really only means adding the relations $[e_s,x_t]=[x_s,e_t]=[x_s,x_t]=0$. In fact $[x_s,x_t]=0$ holds in $\Omega$ too and $[e_s,x_t]=0$ implies $[x_s,e_t]=0$. This can be seen as follows:
\begin{align*}
	T_s T_t &= (-v^{-1} e_s + v(1-e_s) + x_s)(-v^{-1} e_t+v(1-e_t)+x_t) \\
	&= +v^{-2} e_s e_t \\
	&\phantom{\text{= }} - v^{-1} (e_s x_t - x_s e_t) \\
	&\phantom{\text{= }} - e_s(1-e_t) - (1-e_s) e_t + x_s x_t \\
	&\phantom{\text{= }} + v ((1-e_s) x_t + x_s(1-e_t)) \\
	&\phantom{\text{= }} + v^2 e_s e_t \\
\implies \left[T_s,T_t\right] &= (v^{-1}+v)(-[e_s, x_t] +[e_t, x_s]) + [x_s, x_t]
\end{align*}
Therefore $[e_s,x_t]=[e_t,x_s]$ and $[x_s,x_t]=0$ hold in $\Omega$ by construction and adding the relations $[e_s,x_t]=0$ for all $s\in S_1, t\in S_2$ is sufficient to obtain $\Omega_1\otimes\Omega_2$ from $\Omega$.

In other words: $\ker(\tau)=\langle [e_s,x_t] \mid s\in S_1, t\in S_2\rangle$. Multiplying with $E_I$ and $E_J$ we get another generating set for this ideal consisting of
\begin{align*}
	E_I [e_s, x_t] E_J &= E_I (e_s x_t - x_t e_s) E_J \\
	&= e_s X_{IJ}^t - X_{IJ}^t e_s \\
	&=\begin{cases}
	X_{IJ}^t-0 & s\in I, s\notin J \\
	0-X_{IJ}^t & s\notin I, s\in J \\
	X_{IJ}^t-X_{IJ}^t & s\in I, s\in J \\
	0 - 0 & s\notin I, s\notin J
	\end{cases}
\end{align*}
Ob course these elements could already be zero in $\Omega$ and we only need to include those that have a chance of being non-zero in the generating set. So assume $X_{IJ}^t\neq 0$.

In the first case $I\leftarrow J$ is an inclusion edge with both $I_1\supsetneq J_1$ (because $s\in I_1\setminus J_1$) and $I_2\supsetneq J_2$ (because $t\in I_2\setminus J_2$) being proper inclusions. The second case cannot occur because then $s\in J\setminus I$, $t\in I\setminus J$ and $\ord(st)=2$ so that $X_{IJ}=0$ by the $(\alpha)$-relation. That proves the claim.
\end{proof}

\begin{theorem}\label{tensor_morphism:decomposition_psi}
With the notation above, $[\iota_1(\Psi_1),\iota_2(\Psi_2)]=0$ holds in $\Omega$ and therefore $\Psi \isomorphic \Psi_1\otimes\Psi_2$.
\end{theorem}

\begin{proof}
$\iota_1(E_{I_1})$ and $\iota_2(E_{I_2})$ certainly commute with each other.

Let $I\leftrightarrows J$, $K\leftrightarrows L$ be transversal edges in the compatibility graphs $\mathcal{Q}_{W_1}$ and $\mathcal{Q}_{W_2}$ respectively. Then $\iota_1(X_{I J}) = \sum_{B} X_{I\sqcup B, J\sqcup B}$ and $\iota_2(X_{K L}) = \sum_{A} X_{A\sqcup K, A\sqcup L}$ and therefore
\begin{align*}
	\iota_1(X_{I J}) \iota_2(E_{K}) &= \sum_{A,B} X_{I\sqcup B, J\sqcup B} E_{A\sqcup K} \\
		&= X_{I \sqcup K, J\sqcup K} \\
	\iota_2(E_{K}) \iota_1(X_{I J}) &= \sum_{A,B} E_{A\sqcup K} X_{I\sqcup B, J\sqcup B} \\
		&= X_{I \sqcup K, J\sqcup K}
\end{align*}
So that $[\iota_1(\Psi_1),\iota_2(E_{K})]=0$ for all $K\subseteq S_2$. Furthermore the identities
\begin{align*}
	\iota_1(X_{I J}) \iota_2(X_{K L}) &= \sum_{A,B} X_{I\sqcup B, J\sqcup B} \cdot X_{A\sqcup K, A\sqcup L} \\
	&= X_{I \sqcup K, J\sqcup K} X_{J\sqcup K, J\sqcup L} \\
	\iota_2(X_{K L}) \iota_1(X_{I J}) &= \sum_{A,B} X_{A\sqcup K, A\sqcup L} \cdot X_{I\sqcup B, J\sqcup B} \\
	&= X_{I \sqcup K, I\sqcup L} X_{I\sqcup L, J\sqcup L}
\end{align*}
hold. It will now be proven that these two elements are actually the same by virtue of a $(\beta)$-relation. Choose $s\in I\setminus J$ and $t\in K\setminus L$ so that $st=ts$, $X_{I J}=X_{I J}^s$ and $X_{K L} = X_{K L}^t$. Then in $\Omega$ the $(\beta^{st})$-relation reads
\[\sum_{A,B} X_{I\sqcup K,A\sqcup B}^s X_{A\sqcup B,J\sqcup L}^t = \sum_{C,D} X_{I\sqcup K,C\sqcup D}^t X_{C\sqcup D,J\sqcup L}^s.\]

These two sums reduce to only one summand each: Assume first that $X_{I\sqcup K,A\sqcup B}^s$ is an inclusion edge. Then $K\supseteq B \implies L\setminus B \supseteq L\setminus K\neq\emptyset$ because $K\leftrightarrows L$ is transversal. This means that $X_{A\sqcup B,J\sqcup L}^t$ is also transversal and therefore $A=J$ which contradicts $I\supseteq A$ and $I\setminus J\neq\emptyset$. Therefore $X_{I\sqcup K,A\sqcup B}^s$ cannot be an inclusion edge. Similarly none of the other four occuring edge elements on both sides can be an inclusion edge.

Since all four elements are transversal edges, this implies $A=J$, $B=K$, $C=I$ and $D=L$. Therefore only one summand occurs on each side and the relation reduces to
\[X_{I\sqcup K,J\sqcup K}^s X_{J\sqcup K,J\sqcup L}^t = X_{I\sqcup K,I\sqcup L}^t X_{I\sqcup L,J\sqcup L}^s\]
which implies $[\iota_1(X_{IL}),\iota_2(X_{KL})]=0$ for all transversal edges $I\leftrightarrows J$, $K\leftrightarrows L$. Hence $[\iota_1(\Psi_1),\iota_2(\Psi_2)]=0$.

\medbreak
Therefore the inclusions $\iota_1:\Psi_1\to\Psi$, $\iota_2:\Psi_2\to\Psi$ induce a morphism $j:\Psi_1\otimes\Psi_2\to\Psi$. Restricting the morphism $\tau: \Omega\to\Omega_1\otimes\Omega_2$ gives a morphism $\tau_{|\Psi}:\Psi\to\Psi_1\otimes\Psi_2$ and it is easily verified that $\tau_{|\Psi}$ and $j$ are mutually inverse.
\end{proof}

\subsection{The reduction theorem}

We now prove the main theorem of this paper:

\begin{theorem}\label{maintheorem:statement}
If $W_1$ and $W_2$ satisfy the strong $W$-graph decomposition conjecture, then $W=W_1\times W_2$ does too.
\end{theorem}

\begin{remark}
Remember that $\Irr(W)=\Set{\lambda\boxtimes\mu | \lambda\in\Irr(W_1), \mu\in\Irr(W_2)}$ where $\boxtimes$ denotes the exterior tensor product of representations.

The idea of the whole proof is to pretend that the morphism $\tau: \Omega\to \Omega_1\otimes\Omega_2$ is an isomorphism in which case it would be clear that the elements $F^\lambda\otimes F^\mu$ constitute a set of idempotents satisfying the decomposition conjecture. The next best thing to $\Omega=\Omega_1\otimes\Omega_2$ we have at our disposal is the isomorphism $\Psi=\Psi_1\otimes\Psi_2$ from theorem \ref{tensor_morphism:decomposition_psi}.

Of course $\tau$ has a non-trivial kernel and since there are a lot of elements outside $\Psi$ a host of complications arises. The main tool to overcome those is the extensive use of the $(\beta)$-relations
\[ \sum_{J\subseteq S} X_{IJ}^s X_{JK}^t = \sum_{L\subseteq S} X_{IL}^t X_{LK}^s\]
for $I,K\subseteq S$ with $s,t\in I\setminus K$ as a replacement for the commutator relation $[\Omega_1\otimes 1,1\otimes\Omega_2] = 0$. Note that in general there are edge elements corresponding to inclusion edges occurring in these sums which enables us to control the parts outside of $\Psi$ and inside the kernel of $\tau$.
\end{remark}

\begin{lemma}[Existence of idempotents]
If $W_1$ and $W_2$ satisfy (Z1), (Z2) and (Z6), then $W=W_1\times W_2$ satisfies these too by defining
\[F^{\lambda\boxtimes\mu}:=\iota_1(F^\lambda)\iota_2(F^\mu)\]
for all $\lambda\in\Irr(W_1)$ and $\mu\in\Irr(W_2)$.
\end{lemma}
\begin{proof}
Since $W_1$, $W_2$ satisfy (Z6), $F^\lambda$ and $F^\mu$ are contained in $\Psi_1$ and $\Psi_2$ respectively so that $[\iota_1(F^\lambda),\iota_2(F^\mu)]=0$ by Theorem \ref{tensor_morphism:decomposition_psi}. Therefore $F^{\lambda\boxtimes\mu}$ is an idempotent and (Z1) and (Z6) are satisfied.

\medbreak
(Z2) follows from $E_{I_1\sqcup I_2}=\iota_1(E_{I_1})\iota_2(E_{I_2})$ and (Z2) for the factors.
\end{proof}

\begin{lemma}[Partial order]
If $W_1$ and $W_2$ satisfy (Z1), (Z2), (Z6) as well as (Z3), then $W$ also satisfies (Z3).

The partial order $\preceq$ on $\Irr(W)=\Irr(W_1)\times\Irr(W_2)$ can be chosen as the product of $\preceq_1$ and $\preceq_2$, i.e.
\[\lambda\boxtimes\mu \preceq \lambda'\boxtimes\mu' \iff \lambda\preceq_1\lambda' \wedge \mu\preceq_2\mu'\]
\end{lemma}
\begin{proof}
Assume $F^{\lambda\boxtimes\mu} X_{IJ} F^{\lambda'\boxtimes\mu'}\neq 0$. We want to prove $\lambda\preceq_1\lambda'$ and $\mu\preceq_2\mu'$. Consider the following cases:

\medbreak
Case 1: $I\leftrightarrows J$ is transversal.

Case 1a: $I_2=J_2=B$ so that $X_{IJ}=\iota_1(X_{I_1,J_1}) \iota_2(E_B)$ and $I_1\leftrightarrows J_1$ in $\mathcal{Q}_{W_1}$. Therefore
\begin{align*}
	F^{\lambda\boxtimes\mu} X_{IJ} F^{\lambda'\boxtimes\mu'} &= \iota_1(F^\lambda) \iota_1(X_{I_1 J_1}) \iota_1(F^{\lambda'}) \iota_2(F^\mu) \iota_2(E_B) \iota_2(F^{\mu'}) \\
	&= \iota_1(F^\lambda X_{I_1 J_1} F^{\lambda'}) \iota_2(F^\mu E_B F^{\mu'})
\end{align*}
by Theorem \ref{tensor_morphism:decomposition_psi}. Thus $F^\lambda X_{I_1 J_1} F^{\lambda'}\neq 0$ and $F^\mu E_B F^{\mu'}\neq 0$. From this we conclude $ \lambda\preceq_1\lambda'$ and $\mu=\mu'$ because $W_1$ satisfies (Z3) and $W_2$ satisfies (Z1).

\medbreak
Case 1b: $I_1=J_1=A$ similarly implies $\lambda=\lambda'$ and $\mu\preceq_2\mu'$.

\bigbreak
Case 2: $I\supseteq J$ is an inclusion edge.

\medbreak
Case 2a: $I_1\supsetneq J_1$ and $I_2\supsetneq J_2$.

In that case there are $s\in I_1\setminus J_1$ and $t\in I_2\setminus J_2$ so that $X_{IJ}=X_{IJ}^s=X_{IJ}^t$. Therefore $X_{IJ}=E_I \iota_1(X_{I_1 J_1}^s) E_J = E_I \iota_2(X_{I_2 J_2}^t) E_J$ so that
\begin{align*}
F^{\lambda\boxtimes\mu}X_{IJ}F^{\lambda'\boxtimes\mu'} &= E_I \iota_2(F^\mu) \iota_1(F^\lambda X_{I_1 J_1}^s F^{\lambda'}) \iota_2(F^{\mu'}) E_J \\
	&= E_I \iota_1(F^\lambda) \iota_2(F^\mu X_{I_2 J_2}^t F^{\mu'}) \iota_1(F^{\lambda'}) E_J
\end{align*}
which implies $F^\lambda X_{I_1 J_1}^s F^{\lambda'}\neq 0$ and $F^\mu X_{I_2 J_2}^t F^{\mu'}\neq 0$. If (Z3) holds for $W_1$ and $W_2$, then $\lambda\preceq_1\lambda'$ and $\mu\preceq_2\mu'$.

\medbreak
Case 2b: $I_1\supsetneq J_1$ and $I_2=J_2$. (And case 2c: $I_1 \supsetneq J_1$ and $I_2=J_2$ for symmetry reasons)

The first observation is again
\begin{align*}
F^{\lambda\boxtimes\mu} X_{IJ} F^{\lambda'\boxtimes\mu'} &= \iota_2(F^\mu) \iota_1(F^\lambda) \iota_2(E_{I_2}) \iota_1(X_{I_1 J_1}^{s_1})\iota_2(E_{J_2}) \iota_1(F^{\lambda'}) \iota_2(F^{\mu'}) \\
	&= \iota_2(F^\mu E_{I_2}) \iota_1(F^\lambda X_{I_1 J_1} F^{\lambda'}) \iota_2(E_{J_2} F^{\mu'})
\end{align*}
If (Z3) holds for $W_1$, then $\lambda\preceq_1\lambda'$. Now assume $\mu\not\preceq_2\mu'$. We will derive a contradiction from that. The key is the following statement:
\[\tag{A}\label{maintheorem:claim_A}
\mu\not\preceq_2\mu' \implies F^{\lambda\boxtimes\mu} [\iota_1(X_{I_1 J_1}),\iota_2(\Psi_2)]F^{\lambda'\boxtimes\mu'}=0\]

We use the equation
\[\iota_1(X_{I_1 J_1}) = \sum_{A\supseteq B} X_{I_1 \sqcup A, J_1\sqcup B}.\]

Given any $C\subseteq S_2$ we infer
\begin{align*}
	F^{\lambda\boxtimes\mu} \iota_1(X_{I_1 J_1}) \iota_2(E_C) F^{\lambda'\boxtimes\mu'} &= \sum_{A\supseteq C} F^{\lambda\boxtimes\mu} X_{I_1 \sqcup A,J_1\sqcup C} F^{\lambda'\boxtimes\mu'} \\
	&= F^{\lambda\boxtimes\mu} X_{I_1 \sqcup C,J_1\sqcup C} F^{\lambda'\boxtimes\mu'}
\end{align*}
because the occurrence of a summand with $A\supsetneq C$ would imply $\mu\preceq_2\mu'$ by step 2a contrary to our assumption. Analogously $F^{\lambda\boxtimes\mu} \iota_2(E_C) \iota_1(X_{I_1 J_1}) F^{\lambda'\boxtimes\mu'} = F^{\lambda\boxtimes\mu} X_{I_1 \sqcup C,J_1\sqcup C} F^{\lambda'\boxtimes\mu'}$ and therefore $F^{\lambda\boxtimes\mu} [\iota_1(X_{I_1 J_1}),\iota_2(E_C)] F^{\lambda'\boxtimes\mu'} = 0$.

\medbreak
Now let $K\leftrightarrows L$ be any transversal edge in $\mathcal{Q}_{W_2}$. Then
\begin{align*}
	F^{\lambda\boxtimes\mu} \iota_1(X_{I_1 J_1}) \iota_2(X_{KL}) F^{\lambda'\boxtimes\mu'} &= \sum_{A\supseteq B,C} F^{\lambda\boxtimes\mu} X_{I_1 \sqcup A,J_1\sqcup B} X_{C \sqcup K,C\sqcup L} F^{\lambda'\boxtimes\mu'} \\
	&= \sum_{A\supseteq K} F^{\lambda\boxtimes\mu} X_{I_1 \sqcup A,J_1\sqcup K} X_{J_1 \sqcup K,J_1\sqcup L} F^{\lambda'\boxtimes\mu'} \\
	&= \sum_{\substack{A\supseteq K \\ \lambda'',\mu''}} F^{\lambda\boxtimes\mu} X_{I_1 \sqcup A,J_1\sqcup K} F^{\lambda''\boxtimes\mu''} X_{J_1 \sqcup K,J_1\sqcup L} F^{\lambda'\boxtimes\mu'}
\end{align*}
Now $J_1 \sqcup K \leftrightarrows J_1\sqcup L$ is a transversal edge in $\mathcal{Q}_W$ and if a summand is nonzero this implies $\mu''\preceq_2\mu'$ by Step 1b. If one of the summands had a proper inclusion $A\supsetneq K$, then additionally $\mu\preceq_2\mu''$ by Step 2a contrary to our assumption $\mu\not\preceq_2\mu'$. Therefore
\begin{align*}
	F^{\lambda\boxtimes\mu} \iota_1(X_{I_1 J_1}) \iota_2(X_{KL}) F^{\lambda'\boxtimes\mu'} &= \sum_{\lambda'',\mu''} F^{\lambda\boxtimes\mu} X_{I_1 \sqcup K,J_1\sqcup K} F^{\lambda''\boxtimes\mu''} X_{J_1 \sqcup K,J_1\sqcup L} F^{\lambda'\boxtimes\mu'} \\
	&=F^{\lambda\boxtimes\mu} X_{I_1 \sqcup K,J_1\sqcup K} X_{J_1 \sqcup K,J_1\sqcup L} F^{\lambda'\boxtimes\mu'} \tag{1}
\end{align*}
And similarly one verifies
\[F^{\lambda\boxtimes\mu} \iota_2(X_{KL}) \iota_1(X_{I_1 J_1}) F^{\lambda'\boxtimes\mu'} = F^{\lambda\boxtimes\mu}  X_{I_1 \sqcup K,I_1\sqcup L} X_{I_1 \sqcup L,J_1\sqcup L}F^{\lambda'\boxtimes\mu'} \tag{2}.\]

Now we choose $s\in I_1\setminus J_1$ and $t\in K\setminus L$ and apply the $(\beta^{st})$-relation
\[\sum_{A,B} X_{I_1\sqcup K,A\sqcup B}^s X_{A\sqcup B,J_1\sqcup L}^t = \sum_{C,D} X_{I_1\sqcup K,C\sqcup D}^t X_{C\sqcup D,J_1\sqcup L}^s.\]
Consider the left hand side and assume first that $I_1\sqcup K \leftrightarrows A\sqcup B$ is transversal in $\mathcal{Q}_W$. This is only possible if $K=B$ and if $I_1\leftrightarrows A$ in $\mathcal{Q}_{W_1}$. Then $L\setminus B=L\setminus K\neq\emptyset$ and $A\sqcup B \leftrightarrows J_1\sqcup L$ is also transversal in $\mathcal{Q}_W$. This in turn implies $A=J_1$ so that $I_1\leftrightarrows J_1$ in contradiction to our assumption $I_1 \supsetneq J_1$.

Therefore only inclusion edges $I_1\sqcup K \supseteq A \sqcup B$ occur on the left hand side, i.e. $I_1\supseteq A$ and $K\supseteq B$. Again this implies $L\setminus B\supseteq L\setminus K\neq\emptyset$ so that again $A\sqcup B \leftrightarrows J_1\sqcup L$ which implies $A=J_1$ and $B \leftrightarrows L$ in $\mathcal{Q}_{W_2}$.


Now multiply the left hand side with $F^{\lambda\boxtimes\mu}$ and $F^{\lambda'\boxtimes\mu'}$ and insert additional idempotents:
\[\sum_{A,B} F^{\lambda\boxtimes\mu} X_{I_1\sqcup K,A\sqcup B}^s X_{A\sqcup B,J_1\sqcup L}^t F^{\lambda'\boxtimes\mu'} = \sum_{\substack{B,\lambda'',\mu'' \\ K\supseteq B \wedge B \leftrightarrows L}} F^{\lambda\boxtimes\mu} X_{I_1\sqcup K,J_1\sqcup B}^s F^{\lambda''\boxtimes\mu''} X_{J_1\sqcup B,J_1\sqcup L}^t F^{\lambda'\boxtimes\mu'}\]
If a summand is nonzero then $\lambda''=\lambda'$ and $\mu''\preceq_2\mu'$ by Step 1b. If there was a nonzero summand with $B\supsetneq K$ then $\mu\preceq_2\mu''$ by Step 2a contrary to our assumption $\mu\not\preceq_2\mu'$. Therefore the only nonzero summands are those with $B=K$. Thus
\begin{align*}
	\sum_{A,B} F^{\lambda\boxtimes\mu} X_{I_1\sqcup K,A\sqcup B}^s X_{A\sqcup B,J_1\sqcup L}^t F^{\lambda'\boxtimes\mu'} &= \sum_{\lambda'',\mu''} F^{\lambda\boxtimes\mu} X_{I_1\sqcup K,J_1\sqcup K}^s F^{\lambda''\boxtimes\mu''} X_{J_1\sqcup K,J_1\sqcup L}^t F^{\lambda'\boxtimes\mu'} \\
	&= F^{\lambda\boxtimes\mu} X_{I_1\sqcup K,J_1\sqcup K}^s X_{J_1\sqcup K,J_1\sqcup L}^t F^{\lambda'\boxtimes\mu'} \\
	&\overset{(1)}{=} F^{\lambda\boxtimes\mu} \iota_1(X_{I_1 J_1}) \iota_2(X_{KL}) F^{\lambda'\boxtimes\mu'}
\end{align*}
Similarly one simplifies the right hand side of the $(\beta^{st})$-Relation to
\begin{align*}
	\sum_{C,D} F^{\lambda\boxtimes\mu} X_{I_1\sqcup K,C\sqcup D}^t X_{C\sqcup D,J_1\sqcup L}^s F^{\lambda'\boxtimes\mu'} &= F^{\lambda\boxtimes\mu} X_{I_1\sqcup K,I_1\sqcup L}^s X_{I_1\sqcup L,J_1\sqcup L}^t F^{\lambda'\boxtimes\mu'} \\
	&\overset{(2)}{=} F^{\lambda\boxtimes\mu}  \iota_2(X_{KL}) \iota_1(X_{I_1 J_1}) F^{\lambda'\boxtimes\mu'}
\end{align*}
so that $F^{\lambda\boxtimes\mu} [\iota_1(X_{I_1 J_1}), \iota_2(X_{KL})] F^{\lambda'\boxtimes\mu'} = 0$ and by induction one can now show $F^{\lambda\boxtimes\mu} [\iota_1(X_{I_1 J_1}), \iota_2(X_{K_0 K_1} X_{K_1 K_2} \ldots X_{K_{r-1} K_r})] F^{\lambda'\boxtimes\mu'} = 0$ for all $K_0 \leftrightarrows K_1 \leftrightarrows \ldots \leftrightarrows K_r$ in $\mathcal{Q}_{W_2}$. That finally proves \eqref{maintheorem:claim_A}.

Returning to our objective of disproving $\mu \not\preceq_2 \mu'$ we find
\begin{align*}
	0 &\neq F^{\lambda\boxtimes\mu} \iota_1(X_{I_1 J_1}) F^{\lambda'\boxtimes\mu'} \\
	&= F^{\lambda\boxtimes\mu} \iota_1(X_{I_1 J_1}) \iota_2(F^{\mu'}) \iota_1(F^{\lambda'}) \\
	&\overset{\eqref{maintheorem:claim_A}}{=} F^{\lambda\boxtimes\mu} \iota_2(F^{\mu'}) \iota_1(X_{I_1 J_1}) F^{\lambda'\boxtimes\mu'} \\
	&= \iota_1(F^\lambda) \iota_2(F^\mu F^{\mu'}) \iota_1(X_{I_1 J_1}) F^{\lambda'\boxtimes\mu'}
\end{align*}
because $F^{\mu'}\in\Psi_2$ by (Z6). This shows $\mu=\mu'$ because $W_2$ satisfies (Z1) and that is the final contradiction to the assumption $\mu\not\preceq_2\mu'$ and hence (Z3) for $W$ is finally proven.
\end{proof}

\begin{lemma}
If $W_1$ and $W_2$ satisfy (Z1), (Z2), (Z6) as well as (Z5), then $W$ also satisfies (Z5).
\end{lemma}
\begin{proof}
Edge elements for transversal edges $I\leftrightarrows J$ are of course in $\Psi$ by definition. If $I\leftarrow J$ is an inclusion edge, say $I_1\supsetneq J_1$ and $F^{\lambda\boxtimes\mu} X_{IJ} F^{\lambda\boxtimes\mu} \neq 0$, then we repeat an argument from above:
\begin{align*}
	0 &\neq F^{\lambda\boxtimes\mu} X_{IJ} F^{\lambda\boxtimes\mu} \\
	&= \iota_2(F^\mu) \iota_1(F^\lambda) \iota_2(E_{I_2}) \iota_1(X_{I_1 J_1})\iota_2(E_{J_2}) \iota_1(F^{\lambda}) \iota_2(F^{\mu}) \\
	&= \iota_2(F^\mu E_{I_2}) \iota_1(F^\lambda X_{I_1 J_1} F^{\lambda}) \iota_2(E_{J_2} F^{\mu}) \tag{3}\label{maintheorem:Z5}
\end{align*}
If (Z5) holds for $W_1$ then $F^\lambda X_{I_1 J_1} F^{\lambda} \in \Psi_1$ and so the whole term is in $\Psi$. Similarly the same holds for inclusion edges with $I_2\supsetneq J_2$. This proves (Z5) for $W$.
\end{proof}

\begin{lemma}
If $W_1$ and $W_2$ satisfy (Z1) -- (Z6), then $W$ also satisfies (Z4).
\end{lemma}
\begin{proof}
Choose surjections $\psi_\lambda: k^{d_\lambda\times d_\lambda}\twoheadrightarrow F^\lambda k\Omega_1F^\lambda$ and $\psi_\mu: k^{d_\mu\times d_\mu}\twoheadrightarrow F^\mu k\Omega_2 F^\mu$. Then $\psi_\lambda\otimes\psi_\mu$ is a surjection $k^{d_\lambda\times d_\lambda} \otimes k^{d_\mu\times d_\mu} \twoheadrightarrow F^\lambda k\Omega_1 F^\lambda \otimes F^\mu k\Omega_2 F^\mu$.

On the other hand, $\ker(\tau)\cap F^{\lambda\boxtimes\mu} \Omega F^{\lambda\boxtimes\mu} = F^{\lambda\boxtimes\mu} \ker(\tau) F^{\lambda\boxtimes\mu} = 0$ which can be seen as follows:

For all inclusion edges $I\leftarrow J$ with two proper inclusions $I_1\supsetneq J_1$ and $I_2\supsetneq J_2$, we have
\begin{align*}
F^{\lambda\boxtimes\mu} X_{IJ} F^{\lambda\boxtimes\mu} &\overset{\eqref{maintheorem:Z5}}{=} \iota_2(F^\mu E_{I_2}) \iota_1(F^\lambda X_{I_1 J_1} F^{\lambda}) \iota_2(E_{J_2} F^{\mu}) \\
&=  \iota_1(F^\lambda X_{I_1 J_1} F^{\lambda}) \iota_2(F^\mu E_{I_2}) \iota_2( E_{J_2} F^{\mu}) \\
&=  \iota_1(F^\lambda X_{I_1 J_1} F^{\lambda}) \iota_2(F^\mu \underbrace{E_{I_2} E_{J_2}}_{=0} F^{\mu})
\end{align*}
where we have used Theorem \ref{tensor_morphism:decomposition_psi} and (Z5) for $W_1$. Using (Z3) for $W$ and \ref{tensor_morphism:kernel} conclude $F^{\lambda\boxtimes\mu} \ker(\tau) F^{\lambda\boxtimes\mu}=0$.

\medbreak
Therefore $F^{\lambda\boxtimes\mu} k\Omega F^{\lambda\boxtimes\mu} \xrightarrow{\tau} F^\lambda k\Omega_1 F^\lambda \otimes F^\mu k\Omega_2 F^\mu$ is an isomorphism and $\psi_{\lambda\boxtimes\mu}:=\tau^{-1} \circ (\psi_\lambda\otimes\psi_\mu)$ is a surjection $k^{d_{\lambda\boxtimes\mu}\times d_{\lambda\boxtimes\mu}} \twoheadrightarrow F^{\lambda\boxtimes\mu} \Omega F^{\lambda\boxtimes\mu}$.
\end{proof}

\begin{proof}[Proof of the main theorem \ref{maintheorem:statement}]
Combined the preceeding lemmas prove (Z1) -- (Z6) for $W$ from (Z1) -- (Z6) for both $W_1$ and $W_2$.
\end{proof}

\begin{lemma}
If $W_1$ and $W_2$ satisfy (Z1), (Z2), (Z3), (Z5) and (Z6), then $\ker(\tau)$ is a nilpotent ideal, more precisely:

$\ker(\tau)^k=0$ whenever $k>\min\Set{\operatorname{ht}(\Irr(W_1),\preceq_1),\operatorname{ht}(\Irr(W_2),\preceq_2)}$ where $\operatorname{ht}$ denotes the maximal length of a chain in a partial order.
\end{lemma}
\begin{proof}
Consider an inclusion edge $I\supseteq J$, say $I_1\supsetneq J_1$.

Then $X_{IJ} = \iota_2(E_{I_2})\iota_1(X_{I_1 J_1}) \iota_2(E_{J_2})$, (Z5) for $W_1$ and Theorem \ref{tensor_morphism:decomposition_psi} imply
\begin{align*}
F^{\lambda\boxtimes\mu} X_{IJ} F^{\lambda\boxtimes\mu'} &= \iota_2(F^\mu E_{I_2}) \iota_1(F^\lambda X_{I_1 J_1} F^{\lambda}) \iota_2(E_{J_2} F^{\mu'}) \\
&=  \iota_1(F^\lambda X_{I_1 J_1} F^{\lambda}) \iota_2(F^\mu E_{I_2} E_{J_2} F^{\mu'})
\end{align*}
so that we infer $I_2=J_2$ if the left hand side is non-zero.

Analogously $I_2\supsetneq J_2$ and $F^{\lambda\boxtimes\mu} X_{IJ} F^{\lambda'\boxtimes\mu} \neq 0$ implies $I_1=J_1$.

\medbreak
If both inclusions $I_1\supsetneq J_1$ and $I_2\supsetneq J_2$ are proper and $F^{\lambda\boxtimes\mu} X_{IJ} F^{\lambda'\boxtimes\mu'}\neq 0$, then $\lambda\preceq_1\lambda'$ and $\mu\preceq_2\mu'$ by (Z3). And by the above conclusions, neither $\lambda=\lambda'$ nor $\mu=\mu'$ can be true. Therefore these must be proper inequalities.

Because $\ker(\tau)$ is generated by those $X_{IJ}$ with both $I_1\supsetneq J_1$ and $I_2\supsetneq J_2$ by \ref{tensor_morphism:kernel}, this shows that $\ker(\tau)^k=0$ if $k$ is greater than the length of any chain in $\Irr(W_1)$ or $\Irr(W_2)$.
\end{proof}

\bibliography{references}
\end{document}